\documentclass{amsart}

\usepackage{amsmath}
\usepackage{amssymb}
\usepackage{indentfirst}
\usepackage{galois}
\usepackage{amsthm}
\usepackage{amsfonts}
\usepackage{bm}
\usepackage{graphicx}
\usepackage{epstopdf}
\usepackage{mathrsfs}

\newtheorem{theorem}{Theorem}[section]
\newtheorem{proposition}[theorem]{Proposition}
\newtheorem{lemma}[theorem]{Lemma}
\newtheorem{corollary}[theorem]{Corollary}
\newtheorem{observation}[theorem]{Observation}

\newtheorem*{main1}{Main Result 1}
\newtheorem*{main2}{Main Result 2}
\newtheorem*{plu}{Pl\"{u}cker Formula}

\theoremstyle{definition}

\theoremstyle{remark}
\newtheorem{remark}[theorem]{Remark}

\newenvironment{proof1}{{\noindent\it Proof of Main Result 1.}\quad}{\hfill$ \square$ \par}
\newenvironment{proof2}{{\noindent\it Proof of Main Result 2.}\quad}{\hfill$ \square$ \par}

\numberwithin{equation}{section}

\begin{document}
\bibliographystyle{plain}
\title[Numerical Ranges of Composition Operators]{Numerical Ranges of Composition Operators with Elliptic Automorphism Symbols}

\author[Y.X. Gao, Y. Liang, Y. Wang and Z.H. Zhou] {Yong-Xin Gao, Yuxia Liang, Ya Wang and Ze-Hua Zhou$^*$}
\address{\newline  Yong-Xin Gao\newline School of Mathematical Sciences and LPMC, Nankai University, Tianjin 300071, P.R. China.}
\email{tqlgao@163.com}

\address{\newline Yuxia Liang\newline
School of Mathematical Sciences, Tianjin Normal University, Tianjin 300387, P.R. China}
\email{liangyx1986@126.com}

\address{\newline Ya Wang\newline
 Department of Mathematics, Tianjin university of finance and economics, Tianjin 300222, P.R. China.}
\email{wangyasjxsy0802@163.com}

\address{\newline Ze-Hua Zhou\newline School of Mathematics, Tianjin University, Tianjin 300354,
P.R. China.}
\email{zehuazhoumath@aliyun.com; zhzhou@tju.edu.cn}


\date{}
\thanks{\noindent $^*$Corresponding author.\\
This work was supported in part by the National Natural Science
Foundation of China (Grant Nos.12001293, 11701422, 12171353, 11771323).}

\begin{abstract}
In this paper we investigate the numerical ranges of composition operators whose symbols are elliptic automorphisms of finite orders on the Hilbert Hardy space $H^2(D)$.
\end{abstract}

\maketitle

\section{Introduction}

Let $T$ be an operator on a complex Hilbert space $\mathscr{H}$. The numerical range of $T$, denoted by $W(T)$ in this paper, is the image of the unit sphere of
$\mathscr{H}$ under the quadratic form associated with $T$. That is,
$$W(T)=\{\langle Tf, f\rangle : f\in\mathscr{H}, ||f||=1\}.$$
It is a bounded subset of $\mathbb{C}$. And the spectrum of $T$ is contained in the closure of $W(T)$.

In this paper, we discuss the numerical ranges of composition operators on Hardy space. Let $D$ be the open unit disk in $\mathbb{C}$. Recall that the Hardy space $H^2(D)$ consists of the holomorphic functions on $D$ that are square-integrable on unit circle. It is a Hilbert space with the following inner product:
$$\langle f,g\rangle=\lim_{r\to 1^-}\int_{0}^{2\pi}f(re^{i\theta})\overline{g(re^{i\theta})}\frac{d\theta}{2\pi}.$$
It is well known that each holomorphic self-map $\varphi$ of $D$ induces a bounded operator $C_\varphi$ on $H^2(D)$. This operator is defined as $C_\varphi f=f\comp\varphi,$ and is called a composition operator.

The numerical ranges of composition operators have been studied by many experts in the past twenty years. Yet, known results are far from ample, partly because of the difficulties caused by the fact that the numerical range is not invariant under similarity. This makes the study of numerical ranges of composition operators an interesting and challenging topic.

The early attempts focused on the composition operators induced by some concrete self-maps. In 2001, Matache \cite{Mat} figured out the numerical ranges of composition operators induced by monomials. And Shapiro \cite{Inner} gives some information about the numerical ranges of composition operators induced by inner functions.

In 2000, Bourdon and Shapiro considered the numerical ranges of invertible composition operators in their paper \cite{BS}. Recall that a composition operator $C_\varphi$ is invertible on $H^2(D)$ if and only if $\varphi$ is an automorphism of $D$.

If $\varphi$ is a hyperbolic or parabolic automorphism,  Bourdon and Shapiro proof that $W(C_\varphi)$ is a disc centered at the origin. But the radius of the disc and whether the disc is open or closed are unknown, except for a special case.

If $\varphi$ is a elliptic automorphism, then the shape of $W(C_\varphi)$ is closely related to the order of $\varphi$. When the order of $\varphi$ is infinite, Bourdon and Shapiro showed that the closure of $W(C_\varphi)$ is a disc centered at the origin. However, they still did not find out the radius of this disc, and one does not know which points of the boundary, if any, belong to $W(C_\varphi)$.

When $\varphi$ is a elliptic automorphism of finite order, the shape of $W(C_\varphi)$ can become more complicated:

$\bullet$ If $\varphi$ is of order $2$, Bourdon and Shapiro prove that the closure of $W(C_\varphi)$ is an ellipse with foci $\pm 1$. Later in 2005, Abdollahi \cite{Abd05} gave the length of major axis of this ellipse. However, neither \cite{BS} nor \cite{Abd05} gives any information about what points on the boundary of this ellipse belong to $W(C_\varphi)$.

$\bullet$ If $\varphi$ is of finite order greater than $2$, little is known about $W(C_\varphi)$. Bourdon and Shapiro said they `strongly suspect that in this case the closure of $W(C_\varphi)$ is not a disc'. In 2013, Patton proved that Bourdon and Shapiro's conjecture is true when $\varphi$ of order $3$, see \cite{Patton1}. Then in 2015, Heydari and Abdollahi \cite{Abd15} showed that Bourdon and Shapiro's conjecture is true at least `for a large class of finite order elliptic automorphisms'. Recently, Patton and her students verified Bourdon and Shapiro's conjecture when $\varphi$ of order $4$ in their paper \cite{Patton2}.

In this paper we will continue to investigate the numerical range of $C_\varphi$ on $H^2(D)$ when $\varphi$ is an elliptic automorphism of finite order. In Section 4, we will prove that $W(C_\varphi)$ is an open set when $\varphi$ is of order $2$. This result completes the discussion in \cite{BS} and \cite{Abd05}, and is one of our main results in this paper.

\begin{main1}
Suppose $\varphi$ is an elliptic automorphism of order $2$ and the fixed point of  $\varphi$ is $a\in D\backslash\{0\}$. Then the numerical range of $C_\varphi$ on $H^2(D)$ is the open ellipse with foci $\pm 1$ and semi-major axis $\frac{1+|a|^2}{1-|a|^2}$.
\end{main1}

In Section 3 we focus on the case where $\varphi$ is of order $3$. Patton \cite{Patton1} found the support function of $W(C_\varphi)$ for this case. Now in this paper, we will give a complete description of the numerical ranges of composition operators induced by elliptic automorphisms of order 3. It is our second main result here.
\begin{main2}
Suppose $\varphi$ is an elliptic automorphism of order $3$ and the fixed point of  $\varphi$ is $a\in D\backslash\{0\}$. Then the numerical range of $C_\varphi$ on $H^2(D)$ is the interior of the convex hull of an algebraic curve of class $3$ and degree $6$. Moreover, the real foci of this curve are $1$, $e^{2\pi i/3}$ and $e^{4\pi i/3}$. The equation of this curve is given in \text{Section 3} as (\ref{curve*}).
\end{main2}

In particular, the closure of $W(C_\varphi)$ is obviously not a disc when $\varphi$ is of order $3$.

\section{Preliminaries}

\subsection{Support Lines of a Convex Set}
The famous Toeplitz-Hausdorff Theorem, see \cite{Hilbert} for example, states that the numerical range of an operator is always a bounded convex set in $\mathbb{C}$.

Suppose $E$ is a bounded convex set in $\mathbb{C}$. For $\alpha\in\mathbb{R}$, define the line $\mathscr{L}_E(\alpha)$ as follows,
$$\mathscr{L}_E(\alpha) : \cos\alpha\cdot x+\sin\alpha\cdot y-\Lambda_E(\alpha)=0,$$
where
$$\Lambda_E(\alpha)=\sup_{w\in E}{\mathrm{Re}(e^{-i\alpha}w)}.$$
The line $\mathscr{L}_E(\alpha)$ is called the support line of $E$ perpendicular to $e^{i\alpha}$. Obviously, $\mathscr{L}_E(\alpha)$ goes through the point $\Lambda_E(\alpha)e^{i\alpha}$, and it does not separate any two points in $E$. This means that for each $\alpha\in\mathbb{R}$, the set $E$ is contained in one of the halfplanes whose boundaries are $\mathscr{L}_E(\alpha)$. By the Hahn-Banach Theorem, the intersection of these halfplanes is exactly the closure of $E$. So the closure of a convex set is uniquely decided by its support lines.

In our Section 3, in order to determine the convex set $W(C_\varphi)$, we will find out all the support lines of $W(C_\varphi)$. This idea is succeeded from Bourdon and Shapiro's approach in \cite{BS}.

\subsection{Algebraic Curves on Projective Plane}

In this paper we need some basic concepts about the geometry on projective plane. Recall that the projective plane extends the Euclidean plane by adding the `points at infinity'. Homogeneous coordinates $(x,y,z)$ are used to denote the points on projective plane.

Duality between lines and points is a fundamental principle of projective geometry. The equations of lines in homogeneous coordinate are of the form $u_0x+v_0y+w_0z=0$, which are uniquely determined by the homogeneous coordinate of the point $(u_0,v_0,w_0)$ on the dual plane. Conversely, a given point $(x_0,y_0,z_0)$ on the projective plane determines a line $x_0u+y_0v+z_0w=0$ on dual plane.

Suppose $f(x,y,z)=0$ is an algebraic curve on projective plane, then its tangents form a curve in the dual plane, which is called the dual curve of $f(x,y,z)=0$. The equation of the dual curve is called the tangential equation for the original curve. The classic Pl\"{u}cker Formula gives the relation between the degrees of a curve and its dual curve.
\begin{plu}
Suppose $f(x,y,z)=0$ is an algebraic curve of degree $d$, with $\tau$ nodes and $\kappa$ cusps. Then the degree of its dual curve, denoted by $d^*$, satisfies the following equation:
$$d^*=d(d-1)-2\tau-3\kappa.$$
\end{plu}
The degree of the dual curve is called the class of the original curve.

\subsection{Numerical Ranges of Composition Operators}

The space being considered in this paper is the Hardy space $H^2(D)$. It is a Reproducing Kernel Hilbert Space, and the reproducing kernel at point $w\in D$ is
$$K_w(z)=\frac{1}{1-\overline{w}z}.$$
Throughout this paper, we will use the notation $k_w$ to denote the normalized reproducing kernel at point $w\in D$, that is,
$$k_w(z)=K_w(z)/||K_w||=\frac{\sqrt{1-|w|^2}}{1-\overline{w}z}.$$

Several known results about the numerical ranges of composition operators on $H^2(D)$ will be useful in our following discussions. The first result is Corollary 3.4 in \cite{Abd15}.

\begin{theorem}\label{dc}
Suppose $\varphi$ is an elliptic automorphism of order $p$. Let $W(C_\varphi)$ be the numerical range of $C_\varphi$ on $H^2(D)$. Then $\omega\in W(C_\varphi)$ if and only if $e^{2\pi i/p}\omega\in W(C_\varphi)$.
\end{theorem}

The following result is Theorem 7 in \cite{Abd05}, which is the main result there.

\begin{theorem}\label{zc}
Suppose $\varphi$ is an elliptic automorphism of order $2$ and the fixed point of $\varphi$ is $a\in D\backslash\{0\}$. Let $W(C_\varphi)$ be the numerical range of $C_\varphi$ on $H^2(D)$. Then $\overline{W(C_\varphi)}$ is the ellipse with foci $\pm 1$ and semi-major axis $\frac{1+|a|^2}{1-|a|^2}$.
\end{theorem}

Recall that any elliptic automorphism of order $2$ is of the form
$$\varphi_a(z)=\frac{a-z}{1-\overline{a}z}.$$
It is an involution automorphism that exchanges $0$ and $a$. The notation $\varphi_a$ will be used throughout this paper.

\section{Elliptic Automorphisms of Order $3$}

Firstly, we consider the case where $\varphi$ is an elliptic automorphism of order $3$. Instead of computing the numerical range of $C_\varphi$ directly, we deal with its adjoint operator $C_\varphi^*$. From now on, we will always assume that the fixed point of $C_\varphi$ is not zero. Otherwise, $C_\varphi$ will be a monomial, and the corresponding result is given in \cite{Mat}.

\subsection{Eigenvectors of $C_\varphi^*$}

If $\varphi$ is an automorphism of order $3$ with fixed point $a\in D$, then on $H^2(D)$ the operator $C_\varphi^*$ has three eigenvalues: $1$, $\overline{\varphi'(a)}$, and $\overline{\varphi'(a)}^2$. The next lemma gives the eigenvectors corresponding to each eigenvalue. It is a direct corollary of Lemma 2.2 in \cite{BN}, or one can find it as Corollary 2.6 in \cite{MY}.

\begin{lemma}\label{tzz}
Suppose $\varphi$ is an elliptic automorphism of order $3$ with fixed point $a\in D\backslash\{0\}$. Then for $k=0,1,2$, $$\mathrm{Ker}(C_\varphi^*-\overline{\varphi'(a)}^k)=\overline{\mathrm{span}}\{e_{3j+k}-ae_{3j+k-1}; j=0,1,2,...\},$$
where $e_{-1}=0$ and $e_j=k_a\varphi_a^j$ for $j=0,1,2,...$.
\end{lemma}

Note that $\{e_j\}_{j=0}^\infty$ is an orthonormal basis for $H^2(D)$. It is called the Guyker basis. So for each $f\in H^2(D)$, there is a unique decomposition $f=f_1+f_2+f_3$, where $f_k\in\mathrm{Ker}(C_\varphi^*-\overline{\varphi'(a)}^{k-1})$ for $k=1,2,3$.

Now suppose
$$f_1=||f_1||\cdot\left(\alpha_0e_0+\sum_{j=1}^{\infty}\alpha_j\frac{e_{3j}-ae_{3j-1}}{\sqrt{1+|a|^2}}\right)\in \mathrm{Ker}(C_\varphi^*-I);$$
$$f_2=||f_2||\cdot\left(\sum_{j=0}^{\infty}\beta_j\frac{e_{3j+1}-ae_{3j}}{\sqrt{1+|a|^2}}\right)\in\mathrm{Ker}(C_\varphi^*-\overline{\varphi'(a)});$$
$$f_3=||f_3||\cdot\left(\sum_{j=0}^{\infty}\gamma_j\frac{e_{3j+2}-ae_{3j+1}}{\sqrt{1+|a|^2}}\right)\in\mathrm{Ker}(C_\varphi^*-\overline{\varphi'(a)}^2).$$
Here we assume
$$\sum_{j=0}^{\infty}|\alpha_j|^2=\sum_{j=0}^{\infty}|\beta_j|^2=\sum_{j=0}^{\infty}|\gamma_j|^2=1.$$
Then
$$\langle f_2,f_3\rangle=||f_2||||f_3||\cdot\left(-\frac{\overline{a}}{1+|a|^2}\sum_{j=0}^{\infty}\beta_j\overline{\gamma_j}\right);$$
$$\langle f_3,f_1\rangle=||f_3||||f_1||\cdot\left(-\frac{\overline{a}}{1+|a|^2}\sum_{j=0}^{\infty}\gamma_j\overline{\alpha_{j+1}}\right);$$
$$\langle f_1,f_2\rangle=||f_1||||f_2||\cdot\left(-\frac{\overline{a}}{\sqrt{1+|a|^2}}\alpha_0\overline{\beta_0}- \frac{\overline{a}}{1+|a|^2}\sum_{j=1}^{\infty}\alpha_j\overline{\beta_j}\right).$$
From these equalities one can make the following observations.

\begin{observation}\label{ob1}
Suppose $f_k\in\mathrm{Ker}(C_\varphi^*-\overline{\varphi'(a)}^{k-1})\backslash\{0\}$ for $k=1,2,3$. Then
$$\frac{|\langle f_2,f_3\rangle|}{||f_2||||f_3||}\leqslant\frac{|a|}{1+|a|^2};$$
$$\frac{|\langle f_3,f_1\rangle|}{||f_3||||f_1||}\leqslant\frac{|a|}{1+|a|^2};$$
$$\frac{|\langle f_1,f_2\rangle|}{||f_1||||f_2||}\leqslant\frac{|a|}{\sqrt{1+|a|^2}}.$$
\end{observation}

\begin{observation}\label{ob2}
Suppose $f_k\in\mathrm{Ker}(C_\varphi^*-\overline{\varphi'(a)}^{k-1})\backslash\{0\}$ for $k=1,2,3$. Then
$$\left|\frac{\langle f_3,f_1\rangle}{||f_3||||f_1||}\right|^2+\left|\frac{\langle f_1,f_2\rangle}{||f_1||||f_2||}\right|^2<\frac{2|a|^2}{(1+|a|^2)^2}.$$
\end{observation}

\begin{proof}
\begin{align*}
\frac{|\langle f_1,f_2\rangle|}{||f_1||||f_2||}&\leqslant\frac{|a|}{\sqrt{1+|a|^2}}|\alpha_0||\beta_0|+\frac{|a|}{1+|a|^2}\sqrt{1-|\alpha_0|^2}\sqrt{1-|\beta_0|^2}\\
&\leqslant\sqrt{\frac{|a|^2}{1+|a|^2}|\alpha_0|^2+\frac{|a|^2}{(1+|a|^2)^2}(1-|\alpha_0|^2)}.\\
&<\frac{|a|}{1+|a|^2}\sqrt{1+|\alpha_0|^2},
\end{align*}
and
\begin{align*}
\frac{|\langle f_3,f_1\rangle|}{||f_3||||f_1||}&\leqslant\frac{|a|}{1+|a|^2}\sqrt{1-|\alpha_0|^2}.
\end{align*}
So we have
\begin{align*}
\left|\frac{\langle f_3,f_1\rangle}{||f_3||||f_1||}\right|^2+\left|\frac{\langle f_1,f_2\rangle}{||f_1||||f_2||}\right|^2<\frac{2|a|^2}{(1+|a|^2)^2}.
\end{align*}
\end{proof}

\begin{remark}\label{obr}
Since $a$ is the fixed point of $\varphi$ in $D$, Observations \ref{ob1} and \ref{ob2} actually show that $\frac{|\langle f_2,f_3\rangle|}{||f_2||||f_3||}\leqslant\frac{1}{2}$ and $\left|\frac{\langle f_3,f_1\rangle}{||f_3||||f_1||}\right|^2+\left|\frac{\langle f_1,f_2\rangle}{||f_1||||f_2||}\right|^2<\frac{1}{2}$.
\end{remark}

\begin{observation}\label{ob4}
We can never find $f_k\in\mathrm{Ker}(C_\varphi^*-\overline{\varphi'(a)}^{k-1})\backslash\{0\}$ for $k=1,2,3$ such that
$$\frac{|\langle f_2,f_3\rangle|}{||f_2||||f_3||}=\frac{|\langle f_3,f_1\rangle|}{||f_3||||f_1||}=\frac{|\langle f_1,f_2\rangle|}{||f_1||||f_2||}=\frac{|a|}{1+|a|^2}.$$
\end{observation}

\begin{proof}
It is a direct corollary of Observations \ref{ob1} and \ref{ob2}.
\end{proof}

\begin{observation}\label{ob3}
For any $\bm{\theta}=(\theta_1, \theta_2, \theta_3)\in\mathbb{R}^3$ and $\epsilon>0$, one can find $f_k\in\mathrm{Ker}(C_\varphi^*-\overline{\varphi'(a)}^{k-1})\backslash\{0\}$ for $k=1,2,3$, such that
$$0\leqslant e^{i(\theta_{k_1}+\theta_{k_2}-\sum_{k=1}^3\theta_k)}\cdot\frac{\langle f_{k_1},f_{k_2}\rangle}{||f_{k_1}||||f_{k_2}||}+\frac{|a|}{1+|a|^2}<\epsilon$$
for $k_1\ne k_2$.
\end{observation}

\begin{proof}
Take
$$\eta_1=-\mathrm{arg}a-\theta_1;$$
$$\eta_2=-2\mathrm{arg}a-\theta_1-\theta_2;$$
$$\eta_3=-3\mathrm{arg}a-\theta_1-\theta_2-\theta_3.$$
Then let
$\alpha_0=0$ and
$$\alpha_{j+1}=e^{i(\eta_2+j\eta_3)}\sqrt{(1-\rho)\rho^j};$$
$$\beta_j=e^{ij\eta_3}\sqrt{(1-\rho)\rho^j};$$
$$\gamma_j=e^{i(\eta_1+j\eta_3)}\sqrt{(1-\rho)\rho^j},$$
for $j=0,1,2...,$ where $\rho\in(0,1)$. Thus we have
\begin{align*}
\frac{\langle f_2,f_3\rangle}{||f_2||||f_3||}&=-\frac{\overline{a}}{1+|a|^2}\sum_{j=0}^{\infty}\beta_j\overline{\gamma_j}\\
&=-\frac{\overline{a}}{1+|a|^2}\sum_{j=0}^{\infty}e^{ij\eta_3}\sqrt{(1-\rho)\rho^j}e^{-i(\eta_1+j\eta_3)}\sqrt{(1-\rho)\rho^j}\\
&=-e^{i\theta_1}\frac{|a|}{1+|a|^2};
\end{align*}
\begin{align*}
\frac{\langle f_3,f_1\rangle}{||f_3||||f_1||}&=-\frac{\overline{a}}{1+|a|^2}\sum_{j=0}^{\infty}\gamma_j\overline{\alpha_{j+1}}\\
&=-\frac{\overline{a}}{1+|a|^2}\sum_{j=0}^{\infty}e^{i(\eta_1+j\eta_3)}\sqrt{(1-\rho)\rho^j}e^{-i(\eta_2+j\eta_3)}\sqrt{(1-\rho)\rho^j}\\
&=-e^{i\theta_2}\frac{|a|}{1+|a|^2};
\end{align*}
\begin{align*}
\frac{\langle f_1,f_2\rangle}{||f_1||||f_2||}&=-\frac{\overline{a}}{1+|a|^2}\sum_{j=1}^{\infty}\alpha_j\overline{\beta_j}\\
&=-\frac{\overline{a}}{1+|a|^2}\sum_{j=1}^{\infty}e^{i(\eta_2+(j-1)\eta_3)}\sqrt{(1-\rho)\rho^{j-1}}e^{-ij\eta_3}\sqrt{(1-\rho)\rho^j}\\
&=-e^{i\theta_3}\frac{|a|}{1+|a|^2}\rho^{1/2}.
\end{align*}
By letting $\rho\to 1$, we get our conclusion.
\end{proof}

These observations play critical roles in our discussion in the next subsection.

\subsection{The Numerical Range}
From now on, we set $\mu=\overline{\varphi'(a)}$ for convenience. Note that $\mu$ is a third root of unity. Define
\begin{align*}
\mathcal{Q}=\Big\{\bm{\delta}=&(\delta_1,\delta_2,\delta_3)\in\mathbb{R}^3:\exists f_k\in\mathrm{Ker}(C_\varphi^*-\mu^{k-1})\backslash\{0\}\\
 &\mathrm{\; for\;} k=1,2,3, \mathrm{\; s.\; t.\;}\left(\frac{|\langle f_2,f_3\rangle|}{||f_2||||f_3||},\frac{|\langle f_3,f_1\rangle|}{||f_3||||f_1||},\frac{|\langle f_1,f_2\rangle|}{||f_1||||f_2||}\right)=\bm{\delta}\Big\}
\end{align*}
as a subset of $\mathbb{R}^3$. For given $\bm{\delta}=(\delta_1,\delta_2,\delta_3)\in\mathcal{Q}$, let
\begin{align*}
\mathcal{H}_{\bm{\delta}}=\Big{\{}&f=f_1+f_2+f_3\in H^2(D): f_k\in\mathrm{Ker}(C_\varphi^*-\mu^{k-1}),\\
 &\left(|\langle f_2,f_3\rangle|,|\langle f_3,f_1\rangle|,|\langle f_1,f_2\rangle|\right)=(\delta_1||f_2||||f_3||,\delta_2||f_3||||f_1||,\delta_3||f_1||||f_2||)\Big{\}},
\end{align*}
and
$$W_{\bm{\delta}}(C_\varphi^*)=\left\{\langle C_\varphi^*f,f\rangle: ||f||=1, f\in\mathcal{H}_{\bm{\delta}}\right\}.$$

For each $f=f_1+f_2+f_3\in \mathcal{H}_{\bm{\delta}}$, we can find $\bm{\theta}=(\theta_1, \theta_2, \theta_3)\in\mathbb{R}^3$ such that
\begin{align*}
\left(\langle f_2,f_3\rangle,\langle f_3,f_1\rangle,\langle f_1,f_2\rangle\right)=(\delta_1e^{i\theta_1}||f_2||||f_3||,\delta_2e^{i\theta_2}||f_3||||f_1||,\delta_3e^{i\theta_3}||f_1||||f_2||).
\end{align*}
Then
\begin{align}\label{*1}
\langle C_\varphi^*f,f\rangle=&\langle f_1+\mu f_2+\mu^2f_3,f_1+f_2+f_3\rangle\\ \nonumber
=&||f_1||^2-2||f_2||||f_3||\delta_1\cos{(\theta_1-\frac{\pi}{3})}\\ \nonumber
&+\mu\left(||f_2||^2-2||f_1||||f_3||\delta_2\cos{(\theta_2-\frac{\pi}{3})}\right)\\ \nonumber
&+\mu^2\left(||f_3||^2-2||f_1||||f_2||\delta_3\cos{(\theta_3-\frac{\pi}{3})}\right).
\end{align}
At the same time, we have
\begin{align}\label{*2}
||f||^2=&||f_1||^2+||f_2||^2+||f_3||^2+2||f_1||||f_2||\delta_3\cos{\theta_3}\\ \nonumber
&+||f_2||||f_3||\delta_1\cos{\theta_1}+||f_1||||f_3||\delta_2\cos{\theta_2}.
\end{align}

Now for $\alpha\in\mathbb{R}$, let
$$\Lambda=\Lambda(\alpha,\bm{\delta})=\sup_{w\in W_{\bm{\delta}}(C_\varphi^*)}{\mathrm{Re}(e^{-i\alpha}w)},$$
and
$$\Lambda_0(\alpha)=\sup_{\bm{\delta}\in\mathcal{Q}}{\Lambda(\alpha,\bm{\delta})}=\sup_{w\in W(C_\varphi^*)}{\mathrm{Re}(e^{-i\alpha}w)}.$$
We introduce the following partial order on $\mathcal{Q}\subset\mathbb{R}^3$: $\bm{\delta}\leqslant\tilde{\bm{\delta}}$ if $\delta_k\leqslant\tilde{\delta}_k$ for $k=1,2,3$. Moreover, $\bm{\delta}<\tilde{\bm{\delta}}$ if $\bm{\delta}\leqslant\tilde{\bm{\delta}}$ and $\bm{\delta}\ne\tilde{\bm{\delta}}$. Set $\bm{\Delta}=\left(\Delta,\Delta,\Delta\right)$, where $\Delta=\frac{|a|}{1+|a|^2}$. Then by Observations \ref{ob1} and \ref{ob4},
$$\Lambda_0(\alpha)=\max\left\{\sup_{\bm{\delta}<\bm{\Delta}}{\Lambda(\alpha,\bm{\delta})},\sup_{\delta_3>\Delta}{\Lambda(\alpha,\bm{\delta})}\right\}.$$

Note that for each $w=\langle C_\varphi^*f,f\rangle\in W_{\bm{\delta}}(C_\varphi^*)$, by (\ref{*1}) we have
\begin{align}\label{*3}
{\mathrm{Re}(e^{-i\alpha}w)}=&\zeta_1\left(||f_1||^2-2||f_2||||f_3||\delta_1\cos{(\theta_1-\frac{\pi}{3})}\right)\\ \nonumber
&+\zeta_2\left(||f_2||^2-2||f_1||||f_3||\delta_2\cos{(\theta_2-\frac{\pi}{3})}\right)\\ \nonumber
&+\zeta_3\left(||f_3||^2-2||f_1||||f_2||\delta_3\cos{(\theta_3-\frac{\pi}{3})}\right),
\end{align}
where $\zeta_1=\cos\alpha$ and $\{\zeta_k\}_{k=1}^3$ are the roots of the equation $T_3(\zeta)-\cos{3\alpha}=0.$
Here $T_3$ is the Chebyshev polynomial of degree $3$, i.e., $T_3(\zeta)=4\zeta^3-3\zeta$.

For fixed $\alpha\in\mathbb{R}$ and $\bm{\delta}\in\mathcal{Q}$, define a symmetric matrix $M=M(\lambda,\bm{\phi})$ as follows:
\begin{tiny}
\begin{equation*}
\left(\begin{matrix}
\lambda-\zeta_1 & \left(\lambda\cos{\phi_3}+\zeta_3\cos{(\phi_3-\frac{\pi}{3})}\right)\delta_3 & \left(\lambda\cos{\phi_2}+\zeta_2\cos{(\phi_2-\frac{\pi}{3})}\right)\delta_2\\
&&\\
\left(\lambda\cos{\phi_3}+\zeta_3\cos{(\phi_3-\frac{\pi}{3})}\right)\delta_3 & \lambda-\zeta_2 & \left(\lambda\cos{\phi_1}+\zeta_1\cos{(\phi_1-\frac{\pi}{3})}\right)\delta_1\\
&&\\
\left(\lambda\cos{\phi_2}+\zeta_2\cos{(\phi_2-\frac{\pi}{3})}\right)\delta_2 & \left(\lambda\cos{\phi_1}+\zeta_1\cos{(\phi_1-\frac{\pi}{3})}\right)\delta_1 & \lambda-\zeta_3
\end{matrix}\right).
\end{equation*}
\end{tiny}
Here we require that the variable $\lambda$ is positive and the variables $\bm{\phi}=(\phi_1,\phi_2,\phi_3)$ are real. For each $\bm{x}=(x_1,x_2,x_3)\in(\mathbb{R}^+)^3$, we consider $\bm{x}M\bm{x}^T$ as a function with respect to $\bm{\phi}=(\phi_1,\phi_2,\phi_3)$, and assume it achieves its minimum at
$$\bm{\Phi}=\bm{\Phi}(\lambda)=(\Phi_1,\Phi_2,\Phi_3).$$
Then
$$\frac{\partial (\bm{x}M\bm{x}^T)}{\partial\phi_k}=-\frac{2x_1x_2x_3}{x_k}\left(\lambda\sin{\phi_k}+\zeta_k\sin{(\phi_k-\frac{\pi}{3})}\right)\delta_k$$
vanish at $\bm{\phi}=\bm{\Phi}$ for $k=1, 2, 3$.
Note that in order to reach the minimum, it must be required that $\zeta_k\sin\Phi_k\leqslant0$. So
$$\sin\Phi_k=\frac{-\sqrt{3}\zeta_k}{2\sqrt{\lambda^2+\lambda \zeta_k+\zeta_k^2}};$$
$$\cos\Phi_k=\frac{-2\lambda-\zeta_k}{2\sqrt{\lambda^2+\lambda \zeta_k+\zeta_k^2}}.$$
Thus we have
\begin{equation*}
\mathrm{det}M(\lambda,\bm{\Phi})=\left|\begin{matrix}
\lambda-\zeta_1 & -\sqrt{\lambda^2+\lambda \zeta_3+\zeta_3^2}\delta_3 & -\sqrt{\lambda^2+\lambda \zeta_2+\zeta_2^2}\delta_2\\
&&\\
-\sqrt{\lambda^2+\lambda \zeta_3+\zeta_3^2}\delta_3 & \lambda-\zeta_2 & -\sqrt{\lambda^2+\lambda \zeta_1+\zeta_1^2}\delta_1\\
&&\\
-\sqrt{\lambda^2+\lambda \zeta_2+\zeta_2^2}\delta_2 & -\sqrt{\lambda^2+\lambda \zeta_1+\zeta_1^2}\delta_1 & \lambda-\zeta_3
\end{matrix}\right|
\end{equation*}

\begin{equation}\label{det}
=\prod_{j=1}^3(\lambda-\zeta_j)-\sum_{j=1}^3(\lambda^3-\zeta_j^3)\delta_j^2-2\prod_{j=1}^3\sqrt{\lambda^2+\lambda \zeta_j+\zeta_j^2}\cdot\delta_j.
\end{equation}
Let $\lambda=\Lambda'(\alpha,\bm{\delta})$ be the largest positive root of the equation $\mathrm{det}M(\lambda,\bm{\Phi})=0$. The existence of $\Lambda'$ is guaranteed by the proof of the next lemma.

\begin{lemma}
$M(\lambda,\bm{\Phi})$ is positive definite for all $\lambda>\Lambda'$. In particular, $\Lambda'\geqslant\max\{\zeta_1,\zeta_2,\zeta_3\}$.
\end{lemma}

\begin{proof}
Note that by Remark \ref{obr}, $\mathrm{det}M(\lambda,\bm{\Phi})\leqslant 0$ when $\lambda=\max\{\zeta_1,\zeta_2,\zeta_3\}$ and $\mathrm{det}M(\lambda,\bm{\Phi})>0$ for $\lambda$ large enough.
\end{proof}

\begin{corollary}\label{po1}
$M(\lambda,\bm{\phi})$ is positive definite for all $\lambda>\Lambda'$ and $\bm{\phi}\in\mathbb{R}^3$.
\end{corollary}
\begin{proof}
For any nonzero $\bm{x}=(x_1,x_2,x_3)\in\mathbb{R}^3$, take $\tilde{\bm{x}}=(|x_1|,|x_2|,|x_3|)\in(\mathbb{R}^+)^3$. Then
\begin{align*}
\bm{x}M(\lambda,\bm{\phi})\bm{x}^T=\tilde{\bm{x}}M(\lambda,\tilde{\bm{\phi}})\tilde{\bm{x}}^T,
\end{align*}
where $\tilde{\bm{\phi}}=(\phi_1+s_1\pi,\phi_2+s_2\pi,\phi_3+s_3\pi)$ and $s_k=\sum_{l\ne k}\mathrm{sgn} x_l$. Therefore,
\begin{align*}
\bm{x}M(\lambda,\bm{\phi})\bm{x}^T\geqslant\tilde{\bm{x}}M(\lambda,\bm{\Phi})\tilde{\bm{x}}^T>0.
\end{align*}
\end{proof}

So far we have defined two quantities, $\Lambda$ and $\Lambda'$, as functions of $\alpha\in \mathbb{R}$ and $\bm{\delta}\in\mathcal{Q}$. In what follows, we will investigate the relationship between  $\Lambda$ and $\Lambda'$.

\begin{proposition}\label{ppt0}
$\Lambda$ is no larger than $\Lambda'$.
\end{proposition}

\begin{proof}
For arbitrary nonzero $f\in\mathcal{H}_{\bm{\delta}}$ we can write $f=\sum_{k=1}^3f_k$, where $f_k\in\mathrm{Ker}(C_\varphi^*-\mu^{k-1})$ for $k=1,2,3$. Take $\bm{x}=(||f_1||,||f_2||,||f_3||)\in\mathbb{R}^3$ and suppose that
\begin{align*}
\left(\langle f_2,f_3\rangle,\langle f_3,f_1\rangle,\langle f_1,f_2\rangle\right)=(\delta_1e^{i\theta_1}||f_2||||f_3||,\delta_2e^{i\theta_2}||f_3||||f_1||,\delta_3e^{i\theta_3}||f_1||||f_2||).
\end{align*}
Then by (\ref{*1}) and (\ref{*2}) we have
\begin{equation*}
\bm{x}M(\lambda,\bm{\theta})\bm{x}^T=\lambda ||f||^2-\mathrm{Re}(e^{-i\alpha}\langle C_\varphi^*f,f\rangle),
\end{equation*}
where $\bm{\theta}=(\theta_1,\theta_2,\theta_3)$.
According to Corollary \ref{po1}, $M(\Lambda',\bm{\theta})$ is positive semidefinite, so
$$\bm{x}M(\Lambda',\bm{\theta})\bm{x}^T=\Lambda' ||f||^2-\mathrm{Re}(e^{-i\alpha}\langle C_\varphi^*f,f\rangle)\geqslant 0,$$
that is,
$$\Lambda'\geqslant\mathrm{Re}\left(e^{-i\alpha}\left\langle C_\varphi^*\frac{f}{||f||},\frac{f}{||f||}\right\rangle\right).$$
Therefore, by the definition of $\Lambda$ we conclude that $\Lambda\leqslant\Lambda'$.
\end{proof}

\begin{lemma}\label{ppt}
$\Lambda'(\alpha,\bm{\delta})<\Lambda'(\alpha,\tilde{\bm{\delta}})$ whenever $\bm{\delta}<\tilde{\bm{\delta}}$.
\end{lemma}

\begin{proof}
$\lambda=\Lambda'(\alpha,\bm{\delta})$ is a zero of (\ref{det}). Since $\bm{\delta}<\tilde{\bm{\delta}}$, by keeping in mind that $\Lambda'(\alpha,\bm{\delta})\geqslant\max\{\zeta_1,\zeta_2,\zeta_3\}$ we see that
\begin{equation*}
\prod_{j=1}^3(\lambda-\zeta_j)-\sum_{j=1}^3(\lambda^3-\zeta_j^3)\tilde{\delta}_j^2-2\prod_{j=1}^3\sqrt{\lambda^2+\lambda \zeta_j+\zeta_j^2}\cdot\tilde{\delta}_j.
\end{equation*}
is negative at $\lambda=\Lambda'(\alpha,\bm{\delta})$. So according to the definition of $\Lambda'$ we must have $\Lambda'(\alpha,\bm{\delta})<\Lambda'(\alpha,\tilde{\bm{\delta}})$.
\end{proof}

Now take $\Lambda'_0(\alpha)=\Lambda'(\alpha,\bm{\Delta})$. Then $\lambda=\Lambda'_0(\alpha)$ is the largest real zero of
\begin{equation}\label{equa}
\prod_{j=1}^3(\lambda-\zeta_j)-\Delta^2\cdot\sum_{j=1}^3(\lambda^3-\zeta_j^3)-2\Delta^3\cdot\prod_{j=1}^3\sqrt{\lambda^2+\lambda \zeta_j+\zeta_j^2}.
\end{equation}

\begin{remark}\label{dcd}
Since (\ref{equa}) is symmetric with respect to $\zeta_1$, $\zeta_2$ and $\zeta_3$, we have
$$\Lambda'_0(\alpha)=\Lambda'_0(\alpha+\frac{2\pi}{3})=\Lambda'_0(\alpha-\frac{2\pi}{3}).$$
Moreover, Theorem \ref{dc} shows that
$$\Lambda_0(\alpha)=\Lambda_0(\alpha+\frac{2\pi}{3})=\Lambda_0(\alpha-\frac{2\pi}{3}).$$
\end{remark}

\begin{theorem}\label{fan}
For each $\alpha\in\mathbb{R}$ we always have $\Lambda_0=\Lambda'_0$.
\end{theorem}

\begin{proof}
According to Remark \ref{dcd}, without loss of generality we may assume that $\zeta_2\leqslant\zeta_3$. Then we can assert that $\Lambda'(\alpha,\bm{\delta})<\Lambda'_0(\alpha)$ for all $\bm{\delta}\in\mathcal{Q}$. Indeed, if $\bm{\delta}<\bm{\Delta}$, then this claim follows directly from Lemma \ref{ppt}. Otherwise, if $\delta_3>\Delta$, then by Observations \ref{ob1} and \ref{ob2} one can easily check that (\ref{equa}) is negative at $\lambda=\Lambda'(\alpha,\bm{\delta})$. Therefore by the definition of $\Lambda'$ we have
$$\Lambda'(\alpha,\bm{\delta})<\Lambda'(\alpha,\bm{\Delta})=\Lambda'_0(\alpha).$$
Now using Proposition \ref{ppt0}, we get $\Lambda(\alpha,\bm{\delta})<\Lambda'_0(\alpha)$ for all $\bm{\delta}\in\mathcal{Q}$. So $\Lambda_0(\alpha)\leqslant\Lambda'_0(\alpha)$.

One the other hand, by the proof of Observation \ref{ob3}, for arbitrary $\epsilon>0$ and $\bm{\theta}=(\theta_1,\theta_2,\theta_3)\in\mathbb{R}^3$, one can always find $\bm{\delta}\in\mathcal{Q}$ satisfying $|\delta_k-\Delta|<\epsilon$ and unit vectors $f_k\in\mathrm{Ker}(C_\varphi^*-\mu^{k-1})$ such that
\begin{align*}
\left(\langle f_2,f_3\rangle,\langle f_3,f_1\rangle,\langle f_1,f_2\rangle\right)=(\delta_1e^{i\theta_1},\delta_2e^{i\theta_2},\delta_3e^{i\theta_3}).
\end{align*}
For $\bm{x}=(x_1,x_2,x_3)\in\mathbb{R}^3$, let $f=\sum_{k=1}^{3}x_kf_k$. Then $f\in\mathcal{H}_{\bm{\delta}}$. Again by (\ref{*1}) and (\ref{*2}) we have
\begin{equation}\label{repp}
\bm{x}M(\lambda,\bm{\bm{\theta}})\bm{x}^T=\lambda ||f||^2-\mathrm{Re}(e^{-i\alpha}\langle C_\varphi^*f,f\rangle).
\end{equation}
For any $\lambda>\Lambda(\alpha,\bm{\delta})$, the right side of (\ref{repp}) is positive. This means that $M(\lambda,\bm{\theta})$ is positive definite for any $\lambda>\Lambda(\alpha,\bm{\delta})$. Therefore, by the definition of $\Lambda'$ and Proposition \ref{ppt0} we have $$\Lambda'(\alpha,\bm{\delta})=\Lambda(\alpha,\bm{\delta})\leqslant\Lambda_0(\alpha).$$
By letting  $\epsilon$ go to $0$, we get $\Lambda'_0\leqslant\Lambda_0$.
\end{proof}

Theorem \ref{fan} tells us that $\lambda=\Lambda_0$ is actually the largest root of the equation
\begin{equation}\label{eqc}
\prod_{j=1}^3(\lambda-\zeta_j)-\Delta^2\cdot\sum_{j=1}^3(\lambda^3-\zeta_j^3)-2\Delta^3\cdot\prod_{j=1}^3\sqrt{\lambda^2+\lambda \zeta_j+\zeta_j^2}=0.
\end{equation}
Since $\zeta_1$, $\zeta_2$ and $\zeta_3$ are the roots of the equation $4\zeta^3-3\zeta-\cos{3\alpha}=0,$ by noticing that (\ref{eqc}) is symmetric with respect to $\zeta_1$, $\zeta_2$ and $\zeta_3$, it can be written as
\begin{align*}
\lambda^3-\frac{3\lambda}{4}-&\frac{1}{4}\cos{3\alpha}-\left(3\lambda^3-\frac{3}{4}\cos3\alpha\right)\cdot\Delta^2=\\
&2\sqrt{\lambda^6+\frac{3\lambda^4}{4}-\frac{\lambda^3}{2}\cos3\alpha+\frac{9\lambda^2}{16}-\frac{3\lambda}{16}\cos3\alpha+\frac{1}{16}\cos^23\alpha}\cdot\Delta^3.
\end{align*}
By squaring both sides of this equation, we get quadratic equation with respect to $\cos3\alpha$. And then we have
\begin{equation*}
\frac{1}{4}\cos3\alpha=\Lambda_0^3-L\Lambda_0,
\end{equation*}
or
\begin{equation}\label{eq}
\cos^3\alpha-\frac{3}{4}\cos\alpha=\Lambda_0^3-L\Lambda_0,
\end{equation}
where
$$L=\frac{3+6\Delta^3\sqrt{3-3\Delta^2}-6\Delta^4-6\Delta^2}{4(1-\Delta^2)(1-4\Delta^2)}.$$
Notice that $L$ is a constant determined only by the fixed point of $\varphi$. Remember that all support lines of $W(C_\varphi^*)$ are given by
$$\cos\alpha\cdot x+\sin\alpha\cdot y-\Lambda_0(\alpha)=0,$$
so (\ref{eq}) implies that, on the dual plane with homogeneous coordinate $(u,v,w)$, the support lines of $W(C_\varphi^*)$ lie on the curve
\begin{equation}\label{curve}
u^3-3uv^2-4Lu^2w-4Lv^2w+4w^3=0.
\end{equation}

It's a cubic. Thus each support line of $W(C_\varphi^*)$ is actually tangent to the boundary of $W(C_\varphi^*)$, and  $\overline{W(C_\varphi^*)}$ is the convex hull of a curve whose tangential equation is (\ref{curve}). So things need to be done is to figure out the dual curve of (\ref{curve}).
Suppose the equation of the dual curve of (\ref{curve}) is $\Gamma(x,y)=0$ in Euclidean coordinates. One can check that (\ref{curve}) has no singular point, so by the Pl\"{u}cker Formula, $\Gamma$ is of degree $6$.

First of all, the symmetry of (\ref{curve}) implies that $\Gamma$ is a linear combination of
$$\left\{(x^2+y^2)^\nu(x^3-3xy^2)^\tau\right\}_{2\nu+3\tau\leqslant6}.$$
So $\Gamma(x,0)$ is a polynomial of degree $6$ without linear term. It is easy to check that $L>\frac{3}{4}$, and then some more calculations show that one can draw four distinct real tangents of (\ref{curve}) through the point $(0,1,0)$. This means that $\Gamma(x,y)=0$ has four distinct points of intersection with the real axis, whose horizontal ordinates are $\frac{3}{4L}$ and the three roots of the equation $x^3-Lx-\frac{1}{4}=0$. Moreover, since the line $3u+4Lw=0$ is an inflexional tangent to (\ref{curve}), the point $(\frac{3}{4L},0)$ must be a cusp of $\Gamma(x,y)=0$. Therefore, by taking $\Gamma(x,0)$ monic, we have
\begin{align*}
\Gamma(x,0)=&(x-\frac{3}{4L})^3(x^3-Lx-\frac{1}{4})\\
=&x^6-\frac{9}{4L}x^5+(\frac{27}{16L^2}-L)x^4+(2-\frac{27}{64L^3})x^3-\frac{9}{8L}x^2+\frac{27}{256L^3}.
\end{align*}
So
\begin{align*}
\Gamma(x,&y)=P(x^2+y^2)^3+Q(x^3-3xy^2)^2-\frac{9}{4L}(x^2+y^2)(x^3-3xy^2)\\
&+(\frac{27}{16L^2}-L)(x^2+y^2)^2+(2-\frac{27}{64L^3})(x^3-3xy^2)-\frac{9}{8L}(x^2+y^2)+\frac{27}{256L^3},
\end{align*}
where $P+Q=1$.
Finally, by noticing that $(\frac{3}{8L},\frac{1}{\sqrt{L}})$ lies on $\Gamma=0$, we see that
$P=1-\frac{27}{64L^3}$ and $Q=\frac{27}{64L^3}$. So the closure of $W(C_\varphi^*)$ is the convex hull of the curve

\begin{align}\label{curve*}
(1&-\frac{27}{64L^3})(x^2+y^2)^3+\frac{27}{64L^3}(x^3-3xy^2)^2\\
&-\frac{9}{4L}(x^2+y^2)(x^3-3xy^2)+(\frac{27}{16L^2}-L)(x^2+y^2)^2\nonumber\\
&+(2-\frac{27}{64L^3})(x^3-3xy^2)-\frac{9}{8L}(x^2+y^2)+\frac{27}{256L^3}=0.\nonumber
\end{align}

The shape of (\ref{curve}) and (\ref{curve*}) are illuminated in Figure 1 and Figure 2 respectively.

\begin{figure}[htbp]
\begin{minipage}[t]{0.5\linewidth}
\centering
\includegraphics[height=4.5cm]{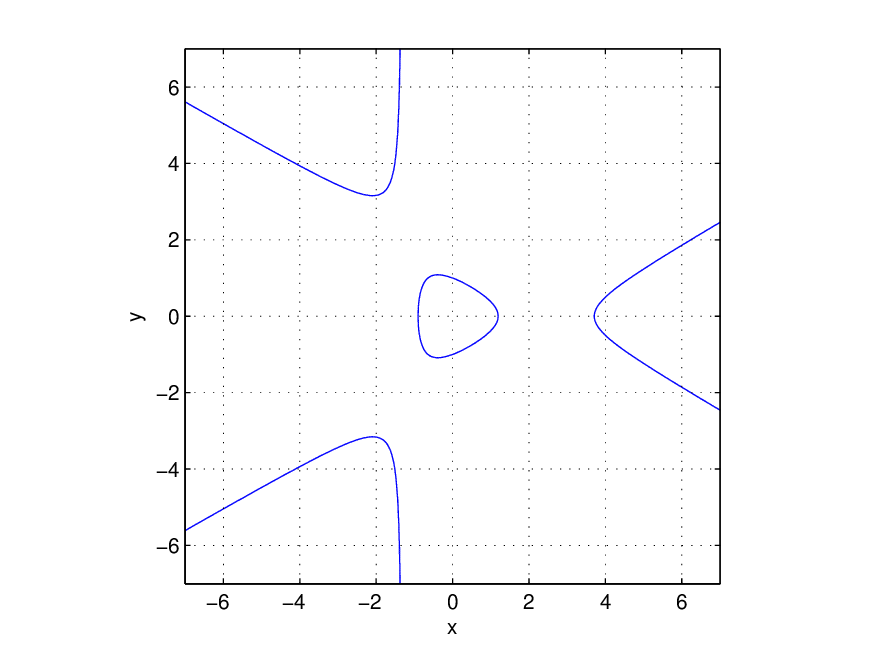}
\caption{Curve (\ref{curve}) when $L=1$.}
\end{minipage}%
\hfill
\begin{minipage}[t]{0.5\linewidth}
\centering
\includegraphics[height=4.5cm]{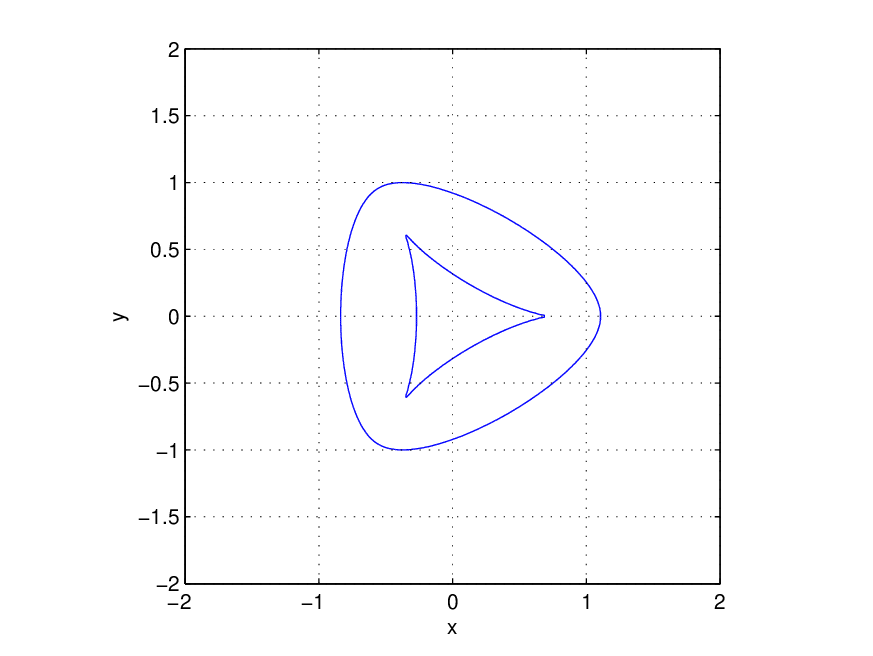}
\caption{Curve (\ref{curve*}) when $L=1$.}
\end{minipage}
\end{figure}

Finally, we show that $W(C_\varphi^*)$ contains no boundary point.

\begin{lemma}
$W(C_\varphi^*)$ is an open set.
\end{lemma}

\begin{proof}
According to Observation \ref{ob4}, for each $\alpha\in\mathbb{R}$, one can never find $f\in H^2(D)$ such that $\mathrm{Re}(e^{-i\alpha}\langle C_\varphi^*f,f\rangle)=\Lambda'_0=\Lambda_0$. So $W(C_\varphi^*)$ contains no boundary point.
\end{proof}

Now, as a conclusion of this section, we can prove our first main result, which gives a precise description of the numerical ranges of composition operators induced by elliptic automorphisms of order $3$.
\\

\begin{proof2}
Note that (\ref{curve*}) is symmetric with respect to the real axis, so the numerical ranges of $C_\varphi$ and $C_\varphi^*$ are exactly the same.

Now we check out the real foci of (\ref{curve*}). For $k=0,1,2$, the lines joining the circular points $(1,i,0)$ and $(\cos{2k\pi i/3},\sin{2k\pi i/3},1)$ lie on (\ref{curve}) on the dual plane. So do the lines joining the other circular points $(1,-i,0)$ and $(\cos{2k\pi i/3},\sin{2k\pi i/3},1)$. Therefore, the real foci of (\ref{curve*}) are $(\cos{2k\pi i/3},\sin{2k\pi i/3},1)$ for  $k=0,1,2$ in homogeneous coordinates.
\end{proof2}

\section{Elliptic Automorphisms of Order $2$}

In this section, we will turn to the elliptic automorphisms of order $2$. In papers \cite{BS} and \cite{Abd05}, it has been shown that if $\varphi$ is an elliptic automorphism of order $2$ with fixed point $a\in D$, then the closure of the numerical range of $C_\varphi$ on $H^2(D)$ is the ellipse with foci $\pm 1$ and semi-major axis $\frac{1+|a|^2}{1-|a|^2}$, see Theorem \ref{zc}. So what we concern here is that if any boundary point of this ellipse belongs to the numerical range of $C_\varphi$.

We want to mention here that the route we followed in the previous section is still available for figuring out $W(C_\varphi)$ when $\varphi$ is an elliptic automorphism of order $2$, only after a slight modification. In fact, the calculation of order $2$ cases is much simpler than what we have done for the order $3$ cases in the last section. However, since the shape of $\overline{W}(C_\varphi)$ has been given in \cite{BS} and \cite{Abd05}, we now adopt a more direct way to show that $W(C_\varphi)$ is actually an open set.

Similar to Lemma \ref{tzz}, the next lemma gives the eigenvector spaces of $C_\varphi^*$ when $\varphi$ is an elliptic automorphisms of order $2$.

\begin{lemma}
Suppose $\varphi$ is an elliptic automorphism of order $2$ with fixed point $a\in D$. Then for $k=0,1$, we have $$\mathrm{Ker}(C_\varphi^*-(-1)^k)=\overline{\mathrm{span}}\{e_{2j+k}-ae_{2j+k-1}; j=0,1,2,...\},$$
where $e_{-1}=0$ and $e_j=k_a\varphi_a^j$ for $j=0,1,2,...$
\end{lemma}

\begin{lemma}\label{lll}
Suppose $\varphi$ is an elliptic automorphism of order $2$ with fixed point $a\in D\backslash\{0\}$. For non-zero vectors $f_k\in\mathrm{Ker}(C_\varphi^*-(-1)^{k-1})$, $k=1,2$ we have
$$\frac{\langle f_1,f_2\rangle}{||f_1||||f_2||}<\frac{2|a|}{1+|a|^2}.$$
\end{lemma}

\begin{proof}

Let
$$f=\sum_{j=0}^{\infty}\beta_j\frac{e_{2j+1}-ae_{2j}}{\sqrt{1+|a|^2}}\in\mathrm{Ker}(C_\varphi^*+1);$$
here we assume $\sum_{j=0}^{\infty}|\beta_j|^2=1,$ so that $||f||=1$.

Then the square of the length of the projection of $f$ in $\mathrm{Ker}(C_\varphi^*-1)$ is
\begin{align*}
&\frac{|a|^2|\beta_0|^2}{1+|a|^2}+\frac{\sum_{j=0}^{\infty}\left|\overline{a}\beta_j+a\beta_{j+1}\right|^2}{(1+|a|^2)^2}\\
<&\frac{|a|^2|\beta_0|^2}{1+|a|^2}+\frac{2\sum_{j=1}^{\infty}|a|^2|\beta_j|^2}{(1+|a|^2)^2}+\frac{2\sum_{j=0}^{\infty}|a|^2|\beta_j|^2}{(1+|a|^2)^2}\\
\leqslant&\frac{4|a|^2}{(1+|a|^2)^2}.
\end{align*}

\end{proof}

Now we can give a proof to our Main Result 1.
\\

\begin{proof1}
By Theorem \ref{zc}, we only need to check that $W(C_\varphi^*)$ is contained in this open ellipse.
For each $f\in H^2(D)$ such that $||f||=1$, we can write $f=f_1+f_2$ where $f_k\in\mathrm{Ker}(C_\varphi^*+(-1)^{k-1})$ for $k=1,2$.
Then
\begin{align*}
\langle C_\varphi^*f,f\rangle=\langle f_1-f_2,f_1+f_2\rangle.
\end{align*}
So
\begin{align*}
1-\langle C_\varphi^*f,f\rangle=&\langle 2f_2,f_1+f_2\rangle\\
=&2||f_2||^2+2\langle f_2,f_1\rangle,
\end{align*}
and
\begin{align*}
1+\langle C_\varphi^*f,f\rangle=&\langle 2f_1,f_1+f_2\rangle\\
=&2||f_1||^2+2\langle f_1,f_2\rangle.
\end{align*}
Hence
\begin{align}\label{21}
\frac{1}{4}|1-\langle C_\varphi^*f,f\rangle|^2-\frac{1}{4}|1+\langle C_\varphi^*f,f\rangle|^2=&(||f_2||^2-||f_1||^2)||f||^2 \nonumber\\
=&||f_2||^2-||f_1||^2.
\end{align}
Suppose that $\langle f_1,f_2\rangle=\delta e^{i\theta}||f_1||\cdot||f_2||$ where $\delta>0$. Then we have
\begin{align*}
\frac{1}{4}|1-\langle C_\varphi^*f,f\rangle|^2=||f_2||^4+||f_1||^2||f_2||^2\delta^2+2||f_1||\cdot||f_2||^3\delta\cos\theta
\end{align*}
and
\begin{align*}
\frac{1}{4}|1+\langle C_\varphi^*f,f\rangle|^2=||f_1||^4+||f_1||^2||f_2||^2\delta^2+2||f_1||^3||f_2||\delta\cos\theta.
\end{align*}
Therefore,

\begin{align}\label{22}
\frac{|1-\langle C_\varphi^*f,f\rangle|-|1+\langle C_\varphi^*f,f\rangle|}{2}\geqslant(||f_2||^2-||f_1||^2)\sqrt{1-\delta^2}.
\end{align}
Combining (\ref{21}) and (\ref{22}) we get
$$\frac{|1-\langle C_\varphi^*f,f\rangle|+|1+\langle C_\varphi^*f,f\rangle|}{2}\leqslant\frac{1}{\sqrt{1-\delta^2}}.$$
Finally, by Lemma \ref{lll}, $\delta<\frac{2|a|}{1+|a|^2}$, so
$$\frac{|1-\langle C_\varphi^*f,f\rangle|+|1+\langle C_\varphi^*f,f\rangle|}{2}<\frac{1+|a|^2}{1-|a|^2}.$$
\end{proof1}

Since each quadratic curve is of class two, it is natural to make the following conjecture about the cases where the order of $\varphi$ is greater than or equal to $4$.

\newtheorem*{conjecture}{Conjecture}

\begin{conjecture}
Suppose $\varphi$ is an elliptic automorphism of finite order $p$ and the fixed point of $\varphi$ is not $0$. Then the numerical range of $C_\varphi$ on $H^2(D)$ is the interior of the convex hull of an algebraic curve of class $p$ and degree $p^2-p$. Moreover, the real foci of the curve are exactly the eigenvalues of $C_\varphi$ on $H^2(D)$, which are $\{e^{2k\pi i/p}\}_{k=1}^p$.
\end{conjecture}

\end{document}